\documentclass[12pt]{amsart}

\usepackage{mathrsfs, amssymb,amsthm, amsfonts, amsmath}
\usepackage[all]{xy}
\usepackage{upgreek}
\usepackage{amsmath}
\usepackage{hyperref}

\usepackage{xcolor}

\newcounter{cequation}[section]

\newtheorem{theorem}[cequation]{Theorem}
\newtheorem*{theorem*}{Theorem}
\newtheorem{lemma}[cequation]{Lemma}
\newtheorem{corollary}[cequation]{Corollary}

\newtheorem{proposition}[cequation]{Proposition}

\theoremstyle{definition}
\newtheorem{example}[cequation]{Example}
\newtheorem{definition}[cequation]{Definition}
\newtheorem*{definition*}{Definition}

\theoremstyle{remark}
\newtheorem{remark}[cequation]{Remark}

\makeatletter\@addtoreset{equation}{section}
\makeatother

\renewcommand\labelenumi{(\arabic{enumi})}

\newcommand{\comm}[1]{} 

\newcommand{\CC}{\mathbb{C}}
\newcommand{\RR}{\mathbb{R}}
\newcommand{\QQ}{\mathbb{Q}}

\newcommand{\PP}{\mathbb{P}}
\newcommand{\KK}{\mathbb{K}}
\newcommand{\LL}{\mathbb{L}}

\newcommand{\SSS}{\mathfrak{S}}
\newcommand{\AAA}{\mathfrak{A}}

\newcommand{\mumu}{{\boldsymbol{\mu}}}

\newcommand{\Aut}{\operatorname{Aut}}
\newcommand{\Bir}{\operatorname{Bir}}

\newcommand{\Gal}{\operatorname{Gal}}

\DeclareMathAlphabet{\mathdutchcal}{U}{dutchcal}{m}{n}
\SetMathAlphabet{\mathdutchcal}{bold}{U}{dutchcal}{b}{n}

\def \ge {\geqslant}
\def \le {\leqslant}

\date{}

\title{Automorphisms of del Pezzo surfaces without points}

\author{Constantin Shramov and Anastasia Vikulova}

\address{\emph{Constantin Shramov}
\newline
\textnormal{Steklov Mathematical Institute of RAS,
8 Gubkina street, Moscow 119991, Russia.
}
\newline
\textnormal{\texttt{costya.shramov@gmail.com}}}

\address{\emph{Anastasia Vikulova}
\newline
\textnormal{Steklov Mathematical Institute of RAS,
8 Gubkina street, Moscow 119991, Russia.
}
\newline
\textnormal{\texttt{vikulovaav@gmail.com}}}

\thanks{This work was supported by the Russian Science Foundation under grant no. 23-11-00033, https://rscf.ru/en/project/23-11-00033/}

\begin{document}

\begin{abstract}
We study automorphism groups of del Pezzo surfaces without points
over a field of zero characteristic, and
estimate their Jordan constants.
\end{abstract}

\maketitle
\tableofcontents

\section{Introduction}

Given a group of geometric origin, it is natural to
try to classify its finite subgroups, since finite groups
are often involved in nice geometric constructions. When an
explicit classification is either impossible or is just too
complicated, the following definition measuring the ``complexity''
of finite groups proved to be useful.

\begin{definition}[{see \cite[Definition~2.1]{Popov},
\cite[Definition~1]{Popov14}}]
\label{definition:Jordan}
A group~$\Gamma$ is called \emph{Jordan}
(alternatively, one says that $\Gamma$ has \emph{Jordan property}),
if there exists a constant $J$ such that any finite
subgroup of $\Gamma$
contains a normal abelian subgroup of index at most~$J$.
If $\Gamma$ is Jordan, we call the minimal such constant $J$
the \emph{Jordan constant} of $\Gamma$ and denote it by~$J(\Gamma)$.
\end{definition}

Jordan property often holds for automorphism and
birational automorphism groups. In particular, it holds for the
birational automorphism group of
the plane over a field of characteristic zero,
see~\mbox{\cite[Theorem~5.3]{Serre-2009}}.
Moreover, the following result was established
in~\mbox{\cite[Theorem~1.9]{Yasinsky}}.

\begin{theorem}
\label{theorem:Egor}
The Jordan constant of the birational automorphism group of
the plane over an algebraically closed
field of characteristic zero equals~$7200$.
\end{theorem}

A generalization of Theorem~\ref{theorem:Egor} for an arbitrary field of characteristic zero
was obtained in~\cite{Zaitsev}.
Furthermore, there exists an analog of Theorem~\ref{theorem:Egor} for forms of $\PP^2$.
Given a variety $X$, by $\Bir(X)$ we will denote the group
of birational automorphisms of $X$.

\begin{theorem}[{see \cite[Theorem~1.2(ii)]{Shramov-SB} and \cite{Shramov-SB-short}}]
\label{theorem:SB}
Let $\KK$ be a field of characteristic zero, and let $S$ be a del Pezzo
surface of degree $9$ over $\KK$. Suppose that~$S$ has no $\KK$-points.
Then
$$
J(\Aut(S))\le J(\Bir(S))\le 3.
$$
Furthermore, these bounds are attained
over a suitable field~$\KK$.
\end{theorem}

\begin{remark}
Over certain fields, a stronger version of Theorem~\ref{theorem:SB} holds.
For instance, if $S$ is a del Pezzo
surface of degree $9$ over the field of rational numbers $\QQ$ such that $S$ has no $\QQ$-points, then
every finite subgroup of $\Bir(S)$ is abelian, see~\mbox{\cite[Corollary~1.4]{Shramov-SB-final}}
or~\mbox{\cite[Theorem~1.3]{Vikulova-SB}}.
\end{remark}

Comparing Theorems~\ref{theorem:Egor} and~\ref{theorem:SB}, we see
that the absence of points over the base field imposes
rather strong restrictions on the birational automorphism group
of a surface, at least in the case of forms of~$\PP^2$.
The same happens in the case of cubic surfaces.

\begin{theorem}[{\cite[Theorem~1.4]{Shramov-Cubics}}]
\label{theorem:cubics}
Let $\KK$ be a field of characteristic zero, and let~\mbox{$S\subset\PP^3$}
be a smooth cubic surface over $\KK$.
Suppose that $S$ has no $\KK$-points.
Then the group~\mbox{$\Aut(S)$} is abelian, and $J(\Bir(S))\le 3$.
Furthermore, the latter bound is attained
over a suitable field~$\KK$.
\end{theorem}

However, for del Pezzo surfaces of degree $2$ the situation is entirely different.
Namely, the largest (and most complicated) possible automorphism
group can be attained on a surface without points. Moreover, in this case
there exists an example when the surface is minimal, i.e. has Picard rank~$1$.

\begin{theorem}[{\cite[Theorem~1.4]{Vikulova-dP}}]
\label{theorem:Vikulova}
There exists a minimal del Pezzo surface~$S$ of degree~$2$ over a
field~$\KK$ of characteristic zero such that $S$ has no $\KK$-points and~\mbox{$J(\Aut(S))=168$}.
\end{theorem}

The goal of this paper is to study automorphism groups
of other del Pezzo surfaces without points
over fields of characteristic
zero, and to bound their Jordan constants.
Our main result is as follows.

\begin{theorem}\label{theorem:main}
Let $\KK$ be a field of characteristic zero, and let $S$ be a del Pezzo
surface of degree $d$ over $\KK$. Suppose that $S$ has no $\KK$-points.
The following assertions hold.
\begin{itemize}
\item[(i)] If $d=2$, then $J(\Aut(S))\le 168$.

\item[(ii)] If $d=4$, then $J(\Aut(S))\le 2$.

\item[(iii)] If $d=6$, then $J(\Aut(S))\le 4$.

\item[(iv)] If $d=8$, then $J(\Aut(S))\le 7200$.
\end{itemize}
Furthermore, each of the above bounds is attained
over a suitable field $\KK$.
\end{theorem}

\begin{remark}
We remind the reader that a del Pezzo surface of degree $1$, $5$, or $7$
always has a point over the base field
(see Theorem~\ref{theorem:dP-with-points} below).
\end{remark}

In certain cases we also obtain bounds for Jordan constants for automorphism groups of minimal
del Pezzo surfaces, which turn out to be stronger than the bounds without minimality assumption,
see Lemmas~\ref{lemma:dP6} and~\ref{lemma:dP8-minimal},
and Examples~\ref{example:dP8-minimal}
and~\ref{example:dP8-minimal-alternative}.
Also, we make the following observation about groups of birational
automorphisms of
minimal del Pezzo surfaces of low degree.

\begin{proposition}\label{proposition:main}
Let $\KK$ be a field of characteristic zero, and let $S$ be a minimal del Pezzo
surface of degree $d$ over $\KK$. Suppose that
$S$ has no $\KK$-points.
The following assertions hold.
\begin{itemize}
\item[(i)] If $d=2$, then $J(\Bir(S))\le 168$.

\item[(ii)] If $d=3$, then the group $\Bir(S)$ is abelian.
\end{itemize}
Furthermore, the bound from assertion~(i)
is attained over a suitable field~$\KK$.
\end{proposition}

It would be interesting to estimate Jordan constants
for automorphism groups and birational automorphism groups of minimal del Pezzo surfaces
of other degrees without points.

\smallskip
The plan of the paper is as follows.
In Section~\ref{section:preliminaries} we collect some useful auxiliary facts
about points on del Pezzo surfaces over algebraically non-closed fields.
In Sections~\ref{section:dP4}, \ref{section:dP6}, and~\ref{section:dP8}
we estimate the Jordan constants of automorphism
groups of del Pezzo surfaces of degrees~$4$, $6$, and~$8$,
respectively. In Section~\ref{section:results} we
complete the proofs of Theorem~\ref{theorem:main}
and Proposition~\ref{proposition:main}.

\smallskip
\textbf{Notation and conventions.}
We denote by $\mumu_n$ the cyclic group of order $n$,
by $\SSS_n$ the symmetric group on $n$ letters,
and by $\AAA_n$ the alternating group on $n$ letters.

Given a field $\KK$, we denote by $\bar{\KK}$ its algebraic closure.
For a variety $X$ defined over $\KK$, we denote by $X_{\bar{\KK}}$
its scalar extension to~$\bar{\KK}$.
By a point of degree $d$ on a variety defined over~$\KK$ we mean a closed point whose
residue field is an extension of~$\KK$ of degree~$d$;
a $\KK$-point is a point of degree~$1$.

A del Pezzo surface is a smooth projective surface with an ample anticanonical class.
For a del Pezzo surface $S$, by its degree we mean its (anti)canonical degree~$K_S^2$.

\smallskip
\textbf{Acknowledgements.}
We are grateful to Andrey Trepalin for many useful discussions.

\section{Preliminaries}
\label{section:preliminaries}

In this section we collect some well known auxiliary facts
concerning points on certain del Pezzo surfaces. The first fact is a general
theorem of Lang and Nishimura.

\begin{theorem}[{see e.g.~\cite[Lemma~1.1]{VA}}]
\label{theorem:Lang-Nishimura}
Let $X$ and $Y$ be smooth projective varieties over a field $\KK$.
Suppose that $X$ is birational to~$Y$. Then $X$ has a $\KK$-point if and only if~$Y$ has a $\KK$-point.
\end{theorem}

\begin{theorem}
\label{theorem:dP-with-points}
Let $S$ be a del Pezzo
surface of degree $d$ over a field~$\KK$.
Suppose that either $d\in\{1,5,7\}$, or $d=8$ and $S_{\bar{\KK}}\not\cong\PP^1\times\PP^1$.
Then $S$ has a $\KK$-point.
\end{theorem}
\begin{proof}
If $d=1$, then the unique base point of the anticanonical linear system $|-K_S|$ is defined over $\KK$.
If $d=5$, then $S$ has a $\KK$-point by
a result going back to Enriques, see for instance \cite[Theorem~2.5]{VA}.
If either $d=7$, or $d=8$ and $S_{\bar{\KK}}\not\cong\PP^1\times\PP^1$, then
$S$ contains a $(-1)$-curve which is defined over $\KK$.
A $(-1)$-curve is a line in the anticanonical embedding of $S$,
so that this line (and hence the surface $S$) contains a $\KK$-point.
\end{proof}

The following theorem is not necessary for the proofs of our results;
however, we mention it to introduce its analogs which
will actually be used.

\begin{theorem}[{see \cite{Coray}}]
\label{theorem:Coray}
Let $S$ be a del Pezzo
surface of degree $d\ge 4$ over a perfect field~$\KK$.
Suppose that $S$ has a point of degree coprime to $d$.
Then $S$ has a $\KK$-point.
\end{theorem}

In some cases, a statement similar to Theorem~\ref{theorem:Coray}
holds for singular analogs of del Pezzo surfaces.

\begin{lemma}[{cf. the proof of \cite[Lemma~2.4]{Shramov-Cubics}}]
\label{lemma:singular-dP-4}
Let $S\subset\PP^4$ be a complete intersection of two quadrics
over an infinite perfect field~$\KK$.
Suppose that $S$ has isolated singularities.
Suppose that~$S$
has a point of degree~$3$ over~$\KK$. Then $S$ has a smooth point defined over~$\KK$.
\end{lemma}

\begin{proof}
Since $S$ is a complete intersection with isolated singularities, it follows
from Hartshorne's connectedness theorem that
$S$ is irreducible, see e.g.~\mbox{\cite[Theorem~18.12]{Eisenbud}}.

Let $P_1$, $P_2$, and $P_3$ be the three points of $S_{\bar{\KK}}$ which form a $\Gal(\bar{\KK}/\KK)$-orbit.
Suppose that the points $P_1$, $P_2$, and $P_3$ are collinear, and
consider the line~$L$
in~$\PP^4$ passing through them. Then $L$
is $\Gal(\bar{\KK}/\KK)$-invariant and thus defined over~$\KK$.
Since it has at least three common points with the surface $S$,
and $S$ is an intersection of quadrics in $\PP^4$,
we conclude that $L$ is contained in $S$.
Since $L$ has an infinite number of $\KK$-points, and the singularities of $S$
are isolated, we see that $S$ has a smooth $\KK$-point.

Therefore, we can assume that the points $P_1$, $P_2$, and $P_3$ are not collinear.
Let $\Pi$ be the plane in $\PP^4$ passing through
these three points. Then $\Pi$
is $\Gal(\bar{\KK}/\KK)$-invariant and thus defined over $\KK$.
Note that $\Pi$ is not contained in $S$, because $S$ is irreducible.
Thus, the intersection $\Lambda=S\cap\Pi$ coincides with an intersection of two curves $C_1, C_2\subset \Pi$
of degree $2$ (defined over~$\bar{\KK}$).

Suppose that $\Lambda$ is one-dimensional.
Then the union of its one-dimensional irreducible components is either a line~$L$,
or a pair of lines $L_1\cup L_2$ (each defined over $\bar{\KK}$),
or a geometrically irreducible conic $C$.
In the first case, the line $L$
is defined over $\KK$, and thus contains
an infinite number of $\KK$-points.
In the other two cases, $\Lambda$ coincides with the curve $L_1\cup L_2$ or $C$, respectively;
in other words, $\Lambda$ has no zero-dimensional irreducible components.
If $\Lambda=L_1\cup L_2$, then none of the points $P_1$, $P_2$, and $P_3$
can coincide with the intersection point $L_1\cap L_2$, because the latter is defined over $\KK$.
Thus, one of the lines $L_1$ or $L_2$ contains at least two of the points $P_i$, while the other
contains at most one. Hence the lines $L_1$ and $L_2$ are defined over $\KK$,
so that $S$ contains an infinite number of $\KK$-points.
If $\Lambda=C$ is a geometrically irreducible conic, then $C$ is defined over $\KK$
and has a point of degree $3$ over $\KK$. This implies that $C$ has a $\KK$-point.
Therefore, one has $C\cong\PP^1$, which means that $C$ (and hence $S$) contains an infinite number of
$\KK$-points. In each of the three cases, we conclude that $S$ has a smooth $\KK$-point, because the
singularities of $S$ are isolated.

Finally, suppose that $\Lambda$ is zero-dimensional.
Then $|\Lambda_{\bar{\KK}}|\le 4$. This implies that the curves $C_1$ and $C_2$ are smooth
at each of the points $P_i$, and intersect each other transversally at these points. Hence
there exists the fourth point $P_4\in \Lambda_{\bar{\KK}}$ different from $P_1$, $P_2$, and~$P_3$; furthermore,
the curves $C_1$ and $C_2$ are smooth and
intersect each other transversally at~$P_4$. So we see that $P_4$ is a smooth point of $S$ defined over~$\KK$.
\end{proof}

Recall that a weak del Pezzo surface is a smooth projective surface
such that its anticanonical class is nef and big. The degree of a weak del Pezzo surface is
defined as the square of its (anti)canonical class.

\begin{lemma}
\label{lemma:weak-dP-4}
Let $S$ be a weak del Pezzo surface of degree $4$ over an infinite perfect field~$\KK$. Suppose that~$S$
has a point of degree~$3$ over~$\KK$. Then $S$ has a $\KK$-point.
\end{lemma}

\begin{proof}
Consider the anticanonical morphism
$$
\psi\colon S\to \PP^4,
$$
and denote its image by $S'$. Then
$\psi$ is a birational
morphism of $S$ to $S'$, and $S'$ is a surface with Du Val singularities.
In particular, the singularities of $S'$ are isolated. Furthermore,~$S'$ is a complete
intersection of two quadrics in~$\PP^4$.

Suppose that the points $P_1$, $P_2$, and $P_3$ are
mapped to one point $R\in S'$ by $\psi$. Then~$R$ is defined over $\KK$, and
$\Gamma=\psi^{-1}(R)$ is a configuration of $(-2)$-curves with the action
of the Galois group $\Gal(\bar{\KK}/\KK)$. Hence the dual graph $\Gamma^\vee$ of $\Gamma$ is a Dynkin diagram
of one of the types $\mathrm{A}_n$ with $n\le 5$, $\mathrm{D}_4$, or $\mathrm{D}_5$,
see~\mbox{\cite[\S8.6.3]{Dolgachev}}.
If $\Gamma^\vee$ is of type $\mathrm{A}_n$ with even $n$, then $\Gamma$, and hence $S$, contains a
$\Gal(\bar{\KK}/\KK)$-invariant point, i.e. a $\KK$-point.
If $\Gamma^\vee$ is either of type $\mathrm{A}_n$ with odd $n$, or of type~$\mathrm{D}_5$, then
$\Aut(\Gamma^\vee)\subset \mumu_2$; since $\Gamma$ contains a $\Gal(\bar{\KK}/\KK)$-orbit~\mbox{$\{P_1, P_2, P_3\}$} of length $3$, we see that
there exists an irreducible component $C$ of $\Gamma$ which is $\Gal(\bar{\KK}/\KK)$-invariant
and contains an odd number of points $P_i$. Since $C_{\bar{\KK}}\cong\PP^1$, this implies that $C$,
and thus also $S$, has a $\KK$-point.
If $\Gamma^\vee\cong \mathrm{D}_4$, then the ``central'' irreducible component~$C$ of~$\Gamma$
is $\Gal(\bar{\KK}/\KK)$-invariant, and contains a $\Gal(\bar{\KK}/\KK)$-invariant subset of cardinality~$3$
formed by the intersections with other irreducible components of~$\Gamma$.
Since~\mbox{$C_{\bar{\KK}}\cong\PP^1$}, we again
conclude that $C$ and $S$ have $\KK$-points.

Therefore, we may assume that $\psi$ maps the points $P_1$, $P_2$, and $P_3$
to different points of~$S'_{\bar{\KK}}$, which means that $S'$ has a point of degree $3$ over $\KK$.
Hence $S'$ contains a smooth $\KK$-point $P$ by
Lemma~\ref{lemma:singular-dP-4}. Since the fiber
$\psi^{-1}(P)_{\bar{\KK}}\subset S_{\bar{\KK}}$ over $P$ consists of one point,
this point is defined over~$\KK$.
\end{proof}

\begin{corollary}\label{corollary:dP6-deg-2-3}
Let $S$ be a del Pezzo surface of degree $6$ over an infinite perfect field $\KK$. Suppose that~$S$
has a point of degree $2$ and a point of degree $3$ over $\KK$. Then $S$ has a $\KK$-point.
\end{corollary}

\begin{proof}
By assumption, the surface $S_{\bar{\KK}}$ contains a $\Gal(\bar{\KK}/\KK)$-orbit
$\{P_1,P_2\}$ of length~$2$.
Suppose that $P_1$ (and thus also $P_2$) is contained in a $(-1)$-curve.
Recall that the dual graph of the configuration of $(-1)$-curves on $S_{\bar{\KK}}$
is a hexagon. If both $P_1$ and $P_2$ are contained in the same $(-1)$-curve, then this
$(-1)$-curve is unique and thus defined over~$\KK$; hence $S$ has a $\KK$-point.
Therefore, we can assume that each $(-1)$-curve on $S_{\bar{\KK}}$ contains at most one point among $P_1$ and $P_2$.
Thus the blow up of $\{P_1,P_2\}$ is a weak del Pezzo surface $\tilde{S}$ of degree $4$ over $\KK$.
It follows from Theorem~\ref{theorem:Lang-Nishimura} that $\tilde{S}$ has a point of degree $3$.
Hence $\tilde{S}$ has a $\KK$-point by Lemma~\ref{lemma:weak-dP-4}.
This implies that $S$ has a $\KK$-point.
\end{proof}

\section{Del Pezzo surfaces of degree $4$}
\label{section:dP4}

In this section we find an estimate for Jordan constants
of automorphism groups of del Pezzo surfaces of degree $4$
without points. First of all, let us introduce some notation.

Let $S$ be a del Pezzo surface of degree $4$ over a field $\KK$ of characteristic zero.
Then~$S_{\bar{\KK}}$ is isomorphic to a blow up of $\PP^2_{\bar{\KK}}$ at five points $P_1, \, P_2, \, P_3, \, P_4, \, P_5$ in general position. Denote by
$$
\pi \colon S_{\bar{\KK}} \to \mathbb{P}^2_{\bar{\KK}}
$$
this blow up.  Denote by
$$
E_1, \ E_2, \ E_3, \ E_4, \ E_5
$$
the preimages under
the morphism $\pi$ of $P_1$, $P_2$, $P_3$, $P_4$, and~$P_5,$  respectively.
Denote by~$Q$  the proper transform under the morphism~$\pi$ of the (smooth)
conic passing through all the points~$P_1,$~$P_2,$~$P_3,$~$P_4$, and~$P_5$.
Denote by $L_{ij}$
for $i,j \in \{1,2,3,4,5\}$, $i\neq j$,
the proper transform under~$\pi$  of the line  passing through~$P_i$ and~$P_j$
(here the index $ij$ is considered to be non-ordered).
Then~$Q$,~$L_{ij}$,  and~$E_i$ give a set of  $16$ lines on~$S_{\bar{\KK}}$;
these are all the lines contained in~$S_{\bar{\KK}}$.
Their intersections are:
\begin{gather*}
E_i^2=Q^2=L_{ij}^2=-1;  \\
E_i \cdot E_j=0 \quad \text{for} \quad  i \neq j;  \\
E_i \cdot Q=L_{ij} \cdot E_i=1 \quad \text{for} \quad  i \neq j;  \\
L_{ij} \cdot Q=0 \quad \text{for} \quad  i \neq j;  \\
L_{ij} \cdot L_{kl}=1,  \quad \text{where} \;  i,j,k,l \in \{1,\ldots,5\} \; \text{are pairwise distinct;}  \\
L_{ij} \cdot L_{jk}=0, \quad \text{where} \;  i,j,k \in \{1,\ldots,5\} \; \text{are pairwise distinct}.
\end{gather*}

The Weyl group
$$
W(\mathrm{D}_5) \cong \mumu_2^4 \rtimes \mathfrak{S}_5
$$
acts on the Picard group~$\mathrm{Pic}(S_{\bar{\KK}})$ by orthogonal transformations with respect to the intersection form. This action fixes the canonical class
$K_{S_{\bar{\KK}}}$ and preserves the cone of effective divisors. The
subgroup $\mathfrak{S}_5 \subset W(\mathrm{D}_5)$ acts on the set of lines on $S_{\bar{\KK}}$ by permutation of indices.
The normal subgroup $\mumu_2^4 \subset W(\mathrm{D}_5)$
consists of two types of involutions~$\iota_{ab}$ and~$\iota_{abcd},$
where $\{a,b\}$ and $\{a,b,c,d\}$ are subsets of cardinality $2$ and $4$
in the set $\{1,2,3,4,5\}$, respectively (thus, the indices $ab$ and $abcd$
are considered to be non-ordered). These involutions act on the set of lines as follows:
\begin{gather*}
\iota_{ab}(E_a)=E_b;\\
\iota_{ab}(E_c)=L_{de}, \quad \text{where} \quad
\{a,b,c,d,e\}=\{1,2,3,4,5\};\\
\iota_{abcd}(L_{ab})=L_{cd};\\
\iota_{abcd}(E_a)=L_{ae} \quad \text{and} \quad \iota_{abcd}(E_e)=Q,
\quad \text{where} \quad \{a,b,c,d,e\}=\{1,2,3,4,5\};
\end{gather*}
we refer the reader to \cite[\S6.4]{DI} for details.
Furthermore, $W(\mathrm{D}_5)$ is the group of \emph{all} automorphisms of the Picard group $\mathrm{Pic}(S_{\bar{\KK}})$
preserving the intersection form and the canonical class, see~\cite[Theorem~23.9]{Manin-CubicForms}.
The action on the lines on $S_{\bar{\KK}}$ provides homomorphisms of the absolute Galois group~$\Gal(\bar{\KK}/\KK)$
and the automorphism group~\mbox{$\Aut(S)$}
to $W(\mathrm{D}_5)$, and the latter homomorphism is injective,
see~\mbox{\cite[Corollary 8.2.40]{Dolgachev}}.
Observe that the images of $\Gal(\bar{\KK}/\KK)$ and $\Aut(S)$ in~$W(\mathrm{D}_5)$ commute with each other.

Let us record some properties of the elements of the Weyl group $W(\mathrm{D}_5)$
which we will need in this section.

\begin{lemma}\label{lemma:WD5}
The following assertions hold.

\begin{enumerate}
\renewcommand\labelenumi{\rm (\roman{enumi})}
\renewcommand\theenumi{\rm (\roman{enumi})}

\item\label{lemma:WD55}
There is a unique conjugacy class of elements of order $5$ in~$W(\mathrm{D}_5)$.

\item\label{lemma:WD53}   There is a unique conjugacy class of elements of order $3$ in~$W(\mathrm{D}_5)$.

\item\label{lemma:WD5-1234}
Let $\rho\colon W(\mathrm{D}_5)\to \mathfrak{S}_5$
be the natural projection. Let $g \in W(\mathrm{D}_5)$ be an element such that
$\rho(g)$ has order $4$.
Then $g$ is conjugate either to $(1234)\in\mathfrak{S}_5$, or to~\mbox{$(1234)\iota_{15}$}.

\item\label{lemma:WD5involutions}    Let $G\subset W(\mathrm{D}_5)$ be a subgroup containing the elements
$\iota_{12}$ and~\mbox{$(1234)$}.
Then the centralizer of $G$ is contained in the group
generated by~$\iota_{1234}$.

\item\label{lemma:WD5involutions15}
Let $G\subset W(\mathrm{D}_5)$ be a subgroup containing the elements
$\iota_{12}$ and~\mbox{$(1234)\iota_{15}$}.
Then the centralizer of $G$ is contained in the group
generated by~$\iota_{1234}$.
\end{enumerate}
\end{lemma}

\begin{proof}
Since 
$$
|W(\mathrm{D}_5)|=2^7\cdot 3\cdot 5,
$$
Sylow theorem implies that every 
element of order $3$ or $5$ in $W(\mathrm{D}_5)$ 
is conjugate to an element of 
a subgroup $\SSS_5\subset W(\mathrm{D}_5)$.
On the other hand, all elements of order $3$ (respectively, of order~$5$)
in $\SSS_5$ are conjugate to each other. This proves assertions~\ref{lemma:WD55} and~\ref{lemma:WD53}.
Assertion~\ref{lemma:WD5-1234} can be found
in~\mbox{\cite[Table 4]{Trepalin-dPfinitefield}}.

Let us prove assertion~\ref{lemma:WD5involutions}.
The centralizer of the element $(1234)$ is generated by $(1234)$
and~$\iota_{1234}$.
Observe that $\iota_{1234}$ commutes with $\iota_{12}$. On the other hand,
one can see that the element $(1234)^k$, $0\le k\le 3$,
commutes with $\iota_{12}$ if and only if $k=0$.
This means that the centralizer of $G$ is contained in the group
generated by~$\iota_{1234}$.

Let us prove assertion~\ref{lemma:WD5involutions15}. We have to find
all the elements $\alpha=\sigma \iota_{ab}$ and $\alpha=\sigma \iota_{abcd}$,
where~\mbox{$\sigma \in \mathfrak{S}_5$},
such that $\alpha$ commutes with $\iota_{12}$ and $(1234)\iota_{15}$.
First of all, note that
$$
\sigma=(1234)^k
$$
for some $0\le k\le 3$,
because these are all the elements of $\mathfrak{S}_5$ which centralize
the image of $G$ in $\mathfrak{S}_5$.
Since $\alpha$ commutes with $\iota_{12}$, we see that $k=0$.
Hence~\mbox{$\alpha=\iota_{ab}$} or~\mbox{$\alpha=\iota_{abcd}$}.
Now the condition that
$\alpha$ commutes with $(1234)\iota_{15}$ implies that either $\alpha$ is
trivial, or~\mbox{$\alpha=\iota_{1234}$}.
\end{proof}

\begin{lemma}\label{lemma:iota12}
Let $\rho\colon W(\mathrm{D}_5)\to \mathfrak{S}_5$
be the natural projection. Let $H \subset \mathrm{Ker}(\rho)$ be a group of order at least $8.$ Assume that the element $(1234) \in W(\mathrm{D}_5)$ normalizes $H.$ Then $\iota_{12} \in H.$
\end{lemma}

\begin{proof}
Assume the contrary. Then
$$
(1234)^k \iota_{12} (1234)^{-k} \notin H
$$
for $0\le k \le 3$.
This means that $\iota_{12}, \iota_{23}, \iota_{34}, \iota_{14} \notin H$.

Let us show that $\iota_{a5} \notin H$ for any $a \in \{1,2,3,4\}$.
Since $(1234)$ normalizes $H$, it is enough to show that
$\iota_{15}\notin H$. Suppose that this is not the case.
Then $\iota_{15}\iota_{12}=\iota_{25}\notin H$,
since $\iota_{12} \notin H$.
However, one has
$$
(1234)^{-1} \iota_{25} (1234)=\iota_{15},
$$
which gives a contradiction.

Since the order of $H$ is at least $8$, and
$$
\iota_{12}, \iota_{23}, \iota_{34}, \iota_{14},
\iota_{15}, \iota_{25}, \iota_{35}, \iota_{45} \notin H,
$$
we see that both $\iota_{1234}$ and $\iota_{1235}$ are contained
in $H$. However, one has
$\iota_{1234}\iota_{1235}=\iota_{45} \notin H$. The obtained contradiction
shows that $\iota_{12} \in H$.
\end{proof}

We will need the following simple group-theoretic lemma.

\begin{lemma}\label{lemma:small-2-group}
Let $G$ be a group which fits into an exact sequence
$$
1\to H\to G\to F\to 1,
$$
where $F\cong\mumu_4$ and $H\cong\mumu_2^2$.
Then $J(G)\le 2$.
\end{lemma}

\begin{proof}
Since $H$ is abelian,
the action of (any) preimages of elements of $F$ on $H$ by conjugation provides a homomorphism
$$
\phi\colon F\to\Aut(H)\cong\mathfrak{S}_3.
$$
Since $F\cong \mumu_4$, the (unique) element $g\in F$ of order $2$ is
contained in the kernel of~$\phi$. Let~\mbox{$\tilde{g}\in G$} be some preimage of $g$.
We conclude that $\tilde{g}$ commutes with $H$, and hence~$\tilde{g}$ and~$H$ generate an abelian subgroup $A$ in $G$.
The index of $A$ equals $2$, and thus~$A$ is normal.
Therefore, we have $J(G)\le 2$.
\end{proof}

The next lemma describes the action of an automorphism on the blow up of
a fixed point.

\begin{lemma}\label{lemma:3actiononP2}
Let $\KK=\bar{\KK}$ be an algebraically closed field of characteristic zero.
Let $\hat{S}$ be a smooth surface over $\KK$, and let $g\in\Aut(\hat{S})$
be an automorphism of finite order $n$. Let~$P$ be a fixed point of $g$, and
suppose that $g$ acts as $\mathrm{diag}(\xi^m, \xi^k)$ on the Zariski tangent
space~\mbox{$T_P(\hat{S})\cong\KK^2$} to $\hat{S}$ at $P$,
where $\xi$ is a primitive $n$-th root of unity,
and~\mbox{$m\not\equiv k\pmod n$}.
Let
$$
\pi\colon S\to \hat{S}
$$
be the blow up of $\hat{S}$ at $P$, and let $E$
be its exceptional divisor.
Then under the natural action of $g$ on $S$ there are exactly two
$g$-fixed points
$R_1$ and $R_2$ on $E$, the action of $g$ on~$T_{R_1}(S)$ is given by
$\mathrm{diag}(\xi^{k-m}, \xi^{m})$, and
the action of $g$ on~$T_{R_2}(S)$ is given
by~\mbox{$\mathrm{diag}(\xi^{m-k}, \xi^{k})$}.
\end{lemma}

\begin{proof}
It is enough to prove the assertion in the case $\KK=\CC$.
Furthermore, there is an analytic neighborhood of $P$ in $\hat{S}$
which can be identified with an analytic neighborhood of~$0$ in~$T_P(\hat{S})$,
and this identification can be chosen equivariant with respect to the group
generated by $g$, see e.g.~\mbox{\cite[\S2.2]{Akhiezer}}.
Therefore, we may assume that
$\hat{S}\cong\mathbb{A}^2$, the point~$P$ is $0\in\mathbb{A}^2$,
and $g=\mathrm{diag}(\xi^m, \xi^k)$ is a linear automorphism.

Let $z$ and $t$ be coordinates on $\hat{S}\cong\mathbb{A}^2$
such that~\mbox{$g(z)=\xi^m z$} and $g(t)=\xi^k t$.
The surface~$S$ is given in $\PP^1 \times \hat{S}$ by
equation  $xt=yz$, where~$x$ and~$y$ are homogeneous coordinates
on~$\PP^1$. The exceptional divisor $E\cong\PP^1$ is identified
with the projectivization of $T_{P}(\hat{S})$,
and thus the action of $g$ on $E$ is
$$
g([x:y])=[\xi^m x: \xi^k y].
$$
Since~\mbox{$m\not\equiv k\pmod n$}, we see that
$g$ has exactly two fixed points on $E$, namely, \mbox{$R_1=[1:0]$}
and $R_2=[0:1]$.

A neighborhood of the point $R_1$ in $S$ is given by equation
\begin{equation*}
t=zy
\end{equation*}
in the affine space $\mathbb{A}^3$ with coordinates $y$, $z$, and $t$.
Thus the Zariski tangent space to $S$ at the point $R_1=(0,0,0)$
is identified with the subspace of $\mathbb{A}^3$
defined by equation~\mbox{$t=0$}.
The element $g$ acts on $y$ by multiplication
by~$\xi^{k-m}$. Hence on $T_{R_1}(S)$ the element $g$ acts
as $\mathrm{diag}(\xi^{k-m}, \xi^{m})$.
The assertion about the action on $T_{R_2}(S)$
is proved by a similar computation.
\end{proof}

Next, we find sufficient conditions for existence of a point on a del Pezzo surface of degree~$4$
in terms of its automorphism group.

\begin{lemma}\label{lemma:dp4order5}
Let $\KK$ be a field of characteristic zero, and let $S$ be a del Pezzo surface of degree $4$ over $\KK$. Suppose that there is an element of order $5$ in $\mathrm{Aut}(S)$. Then there is a $\KK$-point on $S$.
\end{lemma}

\begin{proof}
Let $g \in \mathrm{Aut}(S)$ be an element of order $5$.
By Lemma~\ref{lemma:WD5}\ref{lemma:WD55} we may assume that
$$
g=(12345)\in\mathfrak{S}_5.
$$
Thus, we see that $g$ preserves a unique line on $S_{\bar{\KK}}$. Hence this line is fixed by $\Gal(\bar{\KK}/\KK)$, and
so it is defined over $\KK$. This implies that $S$ has a $\KK$-point.
\end{proof}

\begin{lemma}[{cf.~\cite[proof of Lemma 5.5]{Trepalin-highdegree}}]
\label{lemma:dp4order3}
Let $\KK$ be a field of characteristic zero, and let $S$ be a del Pezzo surface of degree $4$ over $\KK$. Suppose that there is an element of order~$3$ in $\mathrm{Aut}(S)$. Then there is a $\KK$-point on $S$.
\end{lemma}

\begin{proof}
Let $g \in \mathrm{Aut}(S)$ be an element of order $3$. Denote by $G$ the subgroup in~\mbox{$\Aut(S)$} generated by $g$.
By Lemma~\ref{lemma:WD5}\ref{lemma:WD53}  up to conjugation there is a unique element of order~$3$ in $W(\mathrm{D}_5)$.
So we can assume that $g=(123)\in\mathfrak{S}_5$. Thus, the blow down
$$
\pi \colon S_{\bar{\KK}} \to \mathbb{P}^2
$$
of the lines $E_1,$ $E_2,$ $E_3,$ $E_4$, and $E_5$ is $G$-equivariant.
The fixed locus of an element of order $3$ acting regularly on $\mathbb{P}^2$
can be either a union of a point and a line, or a union of three points.
If the fixed locus of $G$ on $\mathbb{P}^2$ is a union of a point and a line,
then we cannot blow up five points in general position,
because in this case any $G$-orbit of length $3$ lies on some line.

So the fixed locus of $G$ on $\mathbb{P}^2$ is a union of three points,
two of which are~\mbox{$P_4=\pi(E_4)$} and~\mbox{$P_5=\pi(E_5)$}.
Denote the other one by $P$.
The action of the group $G$ on $\PP^2$ comes from the projectivization of
the three-dimensional representation of $G$ which is a sum of three
pairwise non-isomorphic one-dimensional representations. Hence
the group $G$ acts on each of the Zariski tangent spaces to $\PP^2$
at the points $P_4,$ $P_5$,
and $P$ as~\mbox{$\mathrm{diag}(\omega, \omega^2)$},
where~$\omega$ is a non-trivial cube root of unity.
According to Lemma~\ref{lemma:3actiononP2},
each of the exceptional divisors~$E_4$ and~$E_5$ contains two
$G$-fixed points,
and $G$ acts on the Zariski tangent space to $S_{\bar{\KK}}$ at each of these
points as $\mathrm{diag}(\omega, \omega)$
or $\mathrm{diag}(\omega^2, \omega^2)$.

Therefore, there are exactly five $G$-fixed points on $S_{\bar{\KK}}$:
four of them are contained in~$E_4$ and~$E_5$,
and $g$ acts on the Zariski tangent spaces at these points
as $\mathrm{diag}(\omega, \omega)$
or~\mbox{$\mathrm{diag}(\omega^2, \omega^2)$};
the fifth point is
$R=\pi^{-1}(P)$, and $g$ acts on $T_R(S)$
as $\mathrm{diag}(\omega, \omega^2)$. Thus, the point~$R$
is $\Gal(\bar{\KK}/\KK)$-invariant,
and hence defined over $\KK$.
\end{proof}

\begin{remark}
For an alternative proof of Lemma~\ref{lemma:dp4order3}, one can show that an automorphism
of order $3$ has exactly five fixed points on~$S_{\bar{\KK}}$, so that
some of these points form a $\Gal(\bar{\KK}/\KK)$-orbit of odd length,
and then apply Theorem~\ref{theorem:Coray}.
\end{remark}

\begin{lemma}\label{lemma:dP4-2-group}
Let $\KK$ be a field of characteristic zero, and let $S$ be a del Pezzo surface of degree $4$ over $\KK$.
Let $\rho\colon W(\mathrm{D}_5)\to \mathfrak{S}_5$ be the natural projection, so that the automorphism
group~\mbox{$\Aut(S)\subset W(\mathrm{D}_5)$} fits into an exact sequence
$$
1\to H\longrightarrow \Aut(S) \stackrel{\rho}{\longrightarrow} F\to 1,
$$
where $F\subset \mathfrak{S}_5$ and $H\cong\mumu_2^k$ for some integer $0\le k\le 4$.
Suppose that $F\cong\mumu_4$ and $k\ge 3$. Then there is a $\KK$-point on~$S$.
\end{lemma}

\begin{proof}
According to Lemma~\ref{lemma:WD5}\ref{lemma:WD5-1234},
up to conjugation we have either~\mbox{$(1234) \in \Aut(S)$},
or~\mbox{$(1234)\iota_{15} \in \Aut(S)$}.
Observe that both of these elements have the same action by conjugation on
$\mumu_2^4\subset W(\mathrm{D}_5)$; thus, in both cases $H$ is normalized
by the element $(1234)$.
So by Lemma~\ref{lemma:iota12} we have $\iota_{12} \in H$.
Now it follows from
Lemma~\ref{lemma:WD5}\ref{lemma:WD5involutions},\ref{lemma:WD5involutions15}
that the image of $\Gal(\bar{\KK}/\KK)$ in $W(\mathrm{D}_5)$
is contained in the group generated by the element $\iota_{1234}$.
Therefore, the pair of intersecting lines~$L_{13}$ and~$L_{24}$
is invariant under the action of the Galois group. This means that their intersection point is defined over~$\KK$.
\end{proof}

Now we are ready to prove the main result of this section.

\begin{lemma}\label{lemma:deg-4-bound}
Let $\KK$ be a field of characteristic zero, and let $S$ be a del Pezzo surface of degree $4$ over $\KK$.
Suppose that
$S$ has no $\KK$-points. Then $J(\Aut(S))\le 2$.
\end{lemma}

\begin{proof}
Since $S$ has no $\KK$-points, we know
from Lemmas~\ref{lemma:dp4order5} and~\ref{lemma:dp4order3} that $\mathrm{Aut}(S)$ does not contain elements of order $5$ and $3$.
So the group $\mathrm{Aut}(S)$ has to be a $2$-group. Consider the projection homomorphism $\rho \colon \mathrm{Aut}(S) \to \mathfrak{S}_5$,
and denote $F=\rho(\Aut(S))$. The kernel of $\rho$ is contained in the subgroup $\mumu_2^4\subset W(\mathrm{D}_5)$.
Thus, there is an exact sequence
$$
0 \to \mumu_2^k \longrightarrow \mathrm{Aut}(S) \stackrel{\rho}{\longrightarrow} F \to 0
$$
for some integer $0 \leqslant k \leqslant 4$.
By~\cite[Theorem~8.6.8]{Dolgachev},
the group $F$ is either trivial, or isomorphic to~$\mumu_2$, or isomorphic to~$\mumu_4$.
If $F$ is trivial, then $\Aut(S)$ is abelian, and hence~\mbox{$J(\Aut(S))=1$}. If~\mbox{$F\cong\mumu_2$}, then $J(\Aut(S)) \leqslant |F|=2$.

Therefore, we may assume that~\mbox{$F\cong\mumu_4$}.
According to Lemma~\ref{lemma:dP4-2-group}, this implies that~\mbox{$k\le 2$}, because $S$ has no $\KK$-points.
If $k=0$ or $1$, then $\Aut(S)$ is a central extension of a cyclic group; thus, it is abelian, so that~\mbox{$J(\Aut(S))=1$}.
Finally, if $k=2$, then~\mbox{$J(\Aut(S))\le 2$} by Lemma~\ref{lemma:small-2-group}.
\end{proof}

Let us show that the bound provided by Lemma~\ref{lemma:deg-4-bound}
is attained.

\begin{example}\label{example:dP4}
Let $S$ be a surface over the field $\RR$
defined in $\PP^4$ with homogeneous coordinates $x$, $y$, $z$, $t$, and $u$ by equations
$$
x^2+y^2+z^2+t^2+u^2=x^2-y^2+2z^2-2t^2=0.
$$
Then $S$ is a del Pezzo surface of degree $4$ without $\RR$-points.
The automorphism group of~$S$ contains a subgroup
$$
G\cong\mumu_2^4\rtimes\mumu_2,
$$
where $\mumu_2^4$ acts by changes of signs of the variables, while the right factor $\mumu_2$
simultaneously interchanges $x$ with $y$ and $z$ with $t$.
Thus
$$
J\big(\Aut(S)\big)\ge J(G)=2.
$$
Therefore, by Lemma~\ref{lemma:deg-4-bound} we have $J(\Aut(S))=2$.
\end{example}

\begin{remark}
A classification of possible automorphism groups of del Pezzo surfaces of
degree $4$ over any field of characteristic zero was obtained in~\cite{Smith}. However,
the approach of~\cite{Smith} does not give a way to figure out which groups act on del Pezzo
surfaces without points over the base field.
\end{remark}

\section{Del Pezzo surfaces of degree $6$}
\label{section:dP6}

In this section we find an estimate for Jordan constants
of automorphism groups of del Pezzo surfaces of degree $6$
without points.

\begin{lemma}\label{lemma:dP6}
Let $S$ be a del Pezzo surface of degree $6$ over a field $\KK$ of characteristic
zero. Suppose that
$S$ has no $\KK$-points. Then $J(\Aut(S))\le 4$.
Furthermore, if $S$ is minimal, then $J(\Aut(S))\le 3$.
\end{lemma}

\begin{proof}
Recall that the dual graph of the configuration of $(-1)$-curves on $S_{\bar{\KK}}$ is a hexagon.
The automorphism group of this configuration
is the dihedral group of order~$12$, which is isomorphic to $\SSS_3\times\mumu_2$.
The action of the group $\Aut(S)$ on
the configuration of $(-1)$-curves provides an exact sequence
$$
1\to T\to \Aut(S)\to \Xi\to 1,
$$
where $T\subset (\bar{\KK}^*)^2$ and $\Xi\subset\SSS_3\times\mumu_2$, see for instance~\mbox{\cite[Theorem~8.4.2]{Dolgachev}}.
Similarly, the Galois group $\Gal(\bar{\KK}/\KK)$ acts on the
configuration of $(-1)$-curves; denote the image of this group in
$\SSS_3\times\mumu_2$ by $\Gamma$. Obviously, the group $\Gamma$ is contained in the centralizer of the group
$\Xi$ in $\SSS_3\times\mumu_2$. Denote by $\sigma$ the non-trivial element of the center $\mumu_2$ of $\SSS_3\times\mumu_2$.

Let us prove the first assertion. If $|\Xi|\le 4$, then the assertion of the lemma follows immediately.
Thus, we need to consider the cases when $\Xi\cong\mumu_6$, $\Xi\cong\SSS_3$, and $\Xi\cong\SSS_3\times\mumu_2$.
(Note that $\SSS_3\times\mumu_2$ has two conjugacy classes of subgroups isomorphic to $\SSS_3$!)

Suppose that $\Xi\cong\SSS_3$ or $\Xi\cong\SSS_3\times\mumu_2$.
Since $\Gamma$ commutes with $\Xi$, we see that $\Gamma$ is contained in the group
$\mumu_2$ generated by $\sigma$. Hence the three contractions $\theta_i$, $1\le i\le 3$, of
pairs of disjoint $(-1)$-curves on $S_{\bar{\KK}}$ which are interchanged by $\sigma$ are all
defined over~$\KK$. Each of them maps $S$ to a product of two conics.
Furthermore, these products are $C_1\times C_2$, $C_2\times C_3$, and $C_3\times C_1$
for some conics $C_1$, $C_2$, and $C_3$ such that the sum of their classes in the Brauer group of $\KK$
equals $0$, cf.~\mbox{\cite[Lemma~3.2]{Trepalin-pointless}}.
\begin{equation}\label{eq:3-contractions}
\xymatrix{
&C_1\times C_2&\\
&S\ar@{->}[u]_{\theta_1}\ar@{->}[ld]_{\theta_2}\ar@{->}[rd]^{\theta_3} & \\
C_2\times C_3&& C_3\times C_1
}
\end{equation}
Since $\Xi\supset\mumu_3$, we see that the group $\Aut(S)$ transitively permutes
the contractions~$\theta_i$. In particular, this implies that
$$
C_1\times C_2\cong C_2\times C_3\cong C_3\times C_1,
$$
so that $C_1\cong C_2\cong C_3$. Together with the condition on the sum of the classes of $C_i$ in
the Brauer group, this gives~\mbox{$C_i\cong\PP^1$}. Hence $C_1\times C_2\cong \PP^1\times\PP^1$ has a $\KK$-point.
Therefore, by Theorem~\ref{theorem:Lang-Nishimura} the surface~$S$ has a $\KK$-point as well, which gives a contradiction.

Now suppose that $\Xi\cong\mumu_6$. Since $\Gamma$ commutes with $\Xi$, we see that $\Gamma$ is contained in~$\Xi$.
Thus there exists a field extension $\KK\subset\LL_3$ of degree $1$ or $3$ such that the three
pairs of disjoint $(-1)$-curves on $S_{\bar{\KK}}$ which are interchanged by $\sigma$ are all
defined over~$\LL_3$. Hence there is a diagram of contractions~\eqref{eq:3-contractions} over $\LL_3$.
The group $\Aut(S)$ transitively permutes the three contractions $\theta_i$, and hence the surface $S$ has a point over $\LL_3$.
Since $S$ has no $\KK$-points, this means that $S$ has a point of degree $3$ over $\KK$.
Furthermore,  there exists a field extension~\mbox{$\KK\subset\LL_2$} of degree $1$ or $2$ such that each of the two
triples of pairwise disjoint $(-1)$-curves on $S_{\bar{\KK}}$ is
defined over~$\LL_2$. The contractions $\eta$ and $\eta^{\mathrm{op}}$ of these triples map~$S_{\LL_2}$ to two Severi--Brauer surfaces
$T$ and $T^{\mathrm{op}}$ defined over $\LL_2$. Furthermore, the sum of the classes of~$T$ and~$T^{op}$ in the Brauer group of~$\LL_2$
equals $0$, see for instance~\mbox{\cite[Lemma~3.1(c)]{IskovskikhTregub}}.
$$
\xymatrix{
&S_{\LL_2}\ar@{->}[ld]_{\eta}\ar@{->}[rd]^{\eta^{\mathrm{op}}}&\\
T&&T^{\mathrm{op}}
}
$$
Since $\Xi\ni\sigma$, the group $\Aut(S)$ interchanges these two contractions, which gives $T\cong T^{\mathrm{op}}$, and hence
$T\cong\PP^2$. Thus $T$ has an $\LL_2$-point, and by Theorem~\ref{theorem:Lang-Nishimura} the surface $S$ has an $\LL_2$-point as well.
Since $S$ has no $\KK$-points, it has a point of degree $2$ over $\KK$.
According to Corollary~\ref{corollary:dP6-deg-2-3}, this implies that $S$ has a $\KK$-point, which gives a contradiction.
Therefore, we conclude that $J(\Aut(S))\le 4$.

Finally, suppose that $S$ is minimal. Then either $\Gamma\cong\mumu_6$, or $\Gamma\cong\SSS_3$, or $\Gamma\cong\SSS_3\times\mumu_2$.
Since $\Xi$ commutes with $\Gamma$, we see that in the former case one has $\Xi\subset\mumu_6$, and in the latter two cases
$\Xi$ is a subgroup of the center $\mumu_2$ of $\SSS_3\times\mumu_2$. However, we already know that the case when $\Xi\cong\mumu_6$ is impossible.
Hence in the former case one has $|\Xi|\le 3$. Therefore, in all possible cases we have
$$
J(\Aut(S))\le |\Xi|\le 3.
$$
\end{proof}

Let us show that the first bound provided by Lemma~\ref{lemma:dP6}
is attained.

\begin{example}\label{example:dP6}
Let $C$ be a conic over the field $\RR$ defined in $\PP^2$
with homogeneous coordinates $x$, $y$, and~$z$ by equation
\begin{equation}\label{eq:xyz}
x^2+y^2+z^2=0.
\end{equation}
Set $\hat{S}=C\times C$, and let
$$
\Sigma=\{(1:\sqrt{-1}:0)\times (1:\sqrt{-1}:0),\quad (1:-\sqrt{-1}:0)\times (1:-\sqrt{-1}:0)\}.
$$
Thus, $\Sigma$ is a $\Gal(\CC/\RR)$-orbit of length $2$ on $C\times C$. Let $S\to\hat{S}$ be the blow up of
$\Sigma$. Then~$S$ is a del Pezzo surface of degree $6$, and $S$ has no $\RR$-points, because $\hat{S}$ has no $\RR$-points.

Observe that for any positive integer $n$ there is an action of the group $\mumu_n$ on $C$ such that a generator of $\mumu_n$ acts as
\begin{equation}\label{eq:xyz-action}
(x:y:z)\mapsto \left(x\cos\frac{2\pi}{n}+y\sin\frac{2\pi}{n}: -x\sin\frac{2\pi}{n}+y\cos\frac{2\pi}{n}:z\right),
\end{equation}
and the points $(1:\sqrt{-1}:0), (1:-\sqrt{-1}:0)\in C(\CC)$
are fixed under this action.
Let the group $G\cong\mumu_n^2\rtimes\mumu_2^2$ act on
$\hat{S}$ so that the generators of the factors of $\mumu_n^2$ act
as in~\eqref{eq:xyz-action}
on the first and the second factor of $C\times C$, respectively;
the non-trivial element~$\iota_1$ of the first $\mumu_2$ acts as
$$
(x_1:y_1:z_1)\times (x_2:y_2:z_2)\mapsto (x_1:-y_1:z_1)\times (x_2:-y_2:z_2),
$$
while the non-trivial element $\iota_2$ of the second $\mumu_2$ acts as
$$
(x_1:y_1:z_1)\times (x_2:y_2:z_2)\mapsto (x_2:y_2:z_2)\times (x_1:y_1:z_1).
$$
Then the set $\Sigma$ is $G$-invariant, and so $G\subset\Aut(S)$.
Observe that, as predicted by Lemma~\ref{lemma:dP6}, one has $J(G)\le 4$, because the index of the normal
abelian subgroup~$\mumu_n^2$ in~$G$ equals~$4$.

Now suppose that $n\ge 5$ is a prime number.
Obviously, the elements $\iota_1$, $\iota_2$, and $\iota_1\iota_2$
do not centralize $\mumu_n^2$. Hence any abelian subgroup $A$ of $G$ which
is not contained in $\mumu_n^2$ intersects
$\mumu_n^2$ by a proper subgroup therein. Since $n$ is prime,
this implies that
$$
\frac{|G|}{|A|}\ge \frac{|\mumu_n^2|}{|A\cap\mumu_n^2|}\ge n>4.
$$
Therefore, we have $J(\Aut(S))=4$.
\end{example}

\section{Del Pezzo surfaces of degree $8$}
\label{section:dP8}

In this section we find an estimate for Jordan constants
of automorphism groups of del Pezzo surfaces of degree $8$.

Since every del Pezzo surface is geometrically rational,
Theorem~\ref{theorem:Egor} implies

\begin{corollary}\label{corollary:dP8}
Let $S$ be a del Pezzo surface over a field of characteristic
zero. Then one has~\mbox{$J(\Aut(S))\le 7200$}.
\end{corollary}

Let us show that the bound provided by Corollary~\ref{corollary:dP8}
is attained.

\begin{example}\label{example:dP8}
Let $C$ be a conic over the field $\RR$ defined in $\PP^2$
by equation~\eqref{eq:xyz}.
Note that $\Aut(C)$ contains the alternating group $\AAA_5$.
Let $S=C\times C$. Then $S$ is a del Pezzo surface of degree $8$
without $\RR$-points, and
$\Aut(S)$ contains a subgroup
$$
G\cong (\AAA_5\times\AAA_5)\rtimes\mumu_2.
$$
Thus
$$
J\big(\Aut(S)\big)\ge J(G)=|G|=7200.
$$
Therefore, by Corollary~\ref{corollary:dP8} we have $J(\Aut(S))=7200$.
\end{example}

For minimal del Pezzo surfaces of degree $8$ (i.e. for those
with Picard rank~$1$), we can
obtain a better bound for the Jordan constant of the automorphism group
than the bound provided by Corollary~\ref{corollary:dP8}.
To do this, we need the following simple fact.

\begin{lemma}\label{lemma:P1}
Let $C$ be a conic over a field $\KK$ of characteristic zero. Then
$$
J\big(\Aut(C)\big)\le 60.
$$
Furthermore,
if a group $\Gamma$ contains~\mbox{$\Aut(C)$} as a subgroup of index $2$,
then~\mbox{$J(\Gamma)\le 120$}.
\end{lemma}

\begin{proof}
The first assertion follows from  the well known classification of finite
subgroups of~\mbox{$\Aut(\PP^1)$}.

To prove the second assertion, let $G\subset \Gamma$ be a finite subgroup.
Recall that the constant $J(G)$ is bounded by a square of the index of any abelian subgroup
in~$G$, see e.g.~\mbox{\cite[Theorem~1.41]{Isaacs}}.
The intersection
$$
G'=G\cap \Aut(C)\subset \Gamma
$$
is a subgroup of index at most $2$ in $G$.
If $G'$ is cyclic or dihedral, then $G'$ contains a cyclic subgroup of index at most $2$;
thus $G$ contains a cyclic subgroup of index at most $4$, which gives
$J(G)\le 4^2=16$. On the other hand,
if $G'$ is one of the three exceptional finite groups
$\mathfrak{A}_4$, $\mathfrak{S}_4$, or $\mathfrak{A}_5$,
then~\mbox{$J(G)\le |G|\le 120$}.
\end{proof}

Using Lemma~\ref{lemma:P1}, we deduce

\begin{lemma}\label{lemma:dP8-minimal}
Let $S$ be a minimal del Pezzo surface of degree $8$ over a field $\KK$
of characteristic zero.
Then $J(\Aut(S))\le 120$.
\end{lemma}

\begin{proof}
Since $S$ is minimal, we have $S_{\bar{\KK}}\cong\PP^1\times\PP^1$.
Moreover, according to \cite[Proposition~7.3]{SV}, the surface $S$
is the Weil restriction of scalars of some conic $C$ defined over a quadratic extension $\LL$ of $\KK$, and
$\Aut(S)$ contains a subgroup of index at most $2$ isomorphic to $\Aut(C)$.
Now the assertion follows from Lemma~\ref{lemma:P1}.
\end{proof}

The bound provided by Lemma~\ref{lemma:dP8-minimal}
is attained for minimal del Pezzo surfaces of degree~$8$ without points over the field of rational numbers~$\QQ$.

\begin{example}\label{example:dP8-minimal}
Let $\LL=\QQ(\sqrt{5})$.
Then a three-dimensional irreducible representation
of the icosahedral group $\AAA_5$ is defined over~$\LL$.
Indeed, the matrices
$$
M_{(12345)}=\frac{1}{4}\left(\begin{array}{ccc}
1+\sqrt{5} & 1-\sqrt{5} & -2\\
1-\sqrt{5} & 2  & -1-\sqrt{5}\\
2 & 1+\sqrt{5} & -1+\sqrt{5}
\end{array}\right),
\quad
M_{(12)(34)}=\left(\begin{array}{ccc}
1 & 0 & 0\\
0 & -1 & 0\\
0 & 0 & -1
\end{array}\right)
$$
satisfy the relations
$$
M_{(12345)}^5=M_{(12)(34)}^2=\big(M_{(12345)}M_{(12)(34)}\big)^3=1,
$$
which imply that they generate a group isomorphic to $\AAA_5$,
see e.g.~\mbox{\cite[Beispiel~19.9]{Huppert}}.
Let $C$ be a conic over the field $\LL$ defined in~$\PP^2$ by equation~\eqref{eq:xyz}.
Then $C$ is the extension of scalars of the conic defined by the same equation~\eqref{eq:xyz}
over~$\QQ$, and $C$ does not contain $\LL$-points.
The group $\AAA_5$ generated by~$M_{(12345)}$ and~$M_{(12)(34)}$ acts on~$C$.
Note that the Galois group
$$
\mathrm{Gal}(\LL/\QQ)\cong\mumu_2
$$
acts on the group $\AAA_5\subset\Aut(C)$ by an involution which is not an inner automorphism.

Let $S$ be the Weil restriction of scalars of $C$. Then
$S$ is a del Pezzo surface of degree~$8$ over $\QQ$, and $S$ does not contain $\QQ$-points.
According to \cite[Proposition~7.3]{SV}, the surface $S$ is minimal, and
$$
\Aut(S)\cong \Aut(C)\rtimes\mumu_2.
$$
The latter group contains the non-trivial semi-direct product $\AAA_5\rtimes\mumu_2\cong\SSS_5$.
Thus
$$
J(\Aut(S))\ge J(\SSS_5)=|\SSS_5|=120.
$$
Therefore, by Lemma~\ref{lemma:dP8-minimal} we have $J(\Aut(S))=120$.
\end{example}

\begin{example}\label{example:dP8-minimal-alternative}
Consider the smooth quadric surface $S$ defined by equations
$$
x+y+z+t+w=x^2+y^2+z^2+t^2+w^2=0
$$
in the projective space $\PP^4$ over $\QQ$ with homogeneous coordinates
$x$, $y$, $z$, $t$, and~$w$.
Obviously, $S$ has no $\QQ$-points.
Furthermore, the group $\SSS_5$ acts on $\PP^4$ by permutations of homogeneous coordinates, and preserves the
surface $S$. Thus, we have
$$
J(\Aut(S))\ge J(\SSS_5)=|\SSS_5|=120.
$$
Observe that $S$ is minimal. Indeed, otherwise
$S$ has two structures of conic bundles defined over~$\QQ$, and a subgroup of $\SSS_5$ of index at most $2$
preserves each of these conic bundles. In other words, the group $\AAA_5$ must act
on a conic defined over $\QQ$, which is impossible, because~$\AAA_5$ has no non-trivial three-dimensional
representations over~$\QQ$.
Therefore, Lemma~\ref{lemma:dP8-minimal} gives~\mbox{$J(\Aut(S))=120$}.

We point out that the surface $S$ is isomorphic to the surface described
in Example~\ref{example:dP8-minimal}. Indeed,
it is minimal over $\QQ$, and its extension of scalars to
some quadratic extension~$\LL$ of~$\QQ$ has Picard rank $2$. This implies that $S$
is a Weil restriction of scalars
of some conic~$C$ defined over the field~$\LL$, and $C$ must be acted on
by a subgroup of index at most $2$ in~$\SSS_5$. Thus, there is an action of
$\AAA_5$ on $C$. Hence $\LL=\QQ(\sqrt{5})$, and $C$ is isomorphic
to the conic defined by equation~\eqref{eq:xyz}, because conics over other quadratic extensions of $\QQ$,
as well as other conics over $\QQ(\sqrt{5})$,
do not contain~$\AAA_5$ in their automorphism groups,
see for instance~\mbox{\cite[Theorem~1.2(c)]{GA}}.
\end{example}

\begin{remark}
According to \cite[Theorem~1.2(4)]{Zaitsev-quadrics}, if $S$ is a smooth quadric surface
with a point over the field~$\QQ$, then $J(\Aut(S))\le 8$.
\end{remark}

\section{Proof of the main results}
\label{section:results}

In this section we complete the proofs of Theorem~\ref{theorem:main}
and Proposition~\ref{proposition:main}.

\begin{proof}[Proof of Theorem~\ref{theorem:main}]
If $d=2$, then
\begin{equation}\label{eq:168}
J(\Aut(S))\le J(\Aut(S_{\bar{\KK}}))\le 168
\end{equation}
by the classification of automorphism groups of del Pezzo surfaces of degree $2$ over algebraically closed fields of characteristic zero,
see~\mbox{\cite[Table~8.9]{Dolgachev}}: the automorphism group of any such surface is a direct product
$\mumu_2\times G$, where $|G|\le 168$.
The bound~\eqref{eq:168} is attained over a suitable field
by Theorem~\ref{theorem:Vikulova}.

For $d=4$, $6$, and $8$ the upper bounds for the Jordan constants are given in
Lemmas~\ref{lemma:deg-4-bound} and~\ref{lemma:dP6},
and Corollary~\ref{corollary:dP8}, respectively. These bounds are attained
over suitable fields according to
Examples~\ref{example:dP4},
\ref{example:dP6}, and~\ref{example:dP8}, respectively.
\end{proof}

\begin{proof}[Proof of Proposition~\ref{proposition:main}]
Observe that a cubic surface without $\KK$-points also has no
points of degree $2$ over $\KK$. On the other hand, it is known that for a minimal del Pezzo
surface~$S$ of degree $d$ over a perfect field $\KK$ such that~$S$ has
no points of degree less than~$d$, one has~\mbox{$\Bir(S)=\Aut(S)$}, see~\mbox{\cite[Theorem~1.6(ii)]{Iskovskikh96}}.
Together with Theorems~\ref{theorem:cubics} and~\ref{theorem:main}(i), this implies the required assertions.
\end{proof}

\end{document}